\documentclass[12pt,a4paper]{amsart}
\usepackage{graphicx,amssymb}
\input xy
\xyoption{all}
\usepackage[all]{xy}
\usepackage{hyperref}

\newtheorem{thm}{Theorem}[section]
\newtheorem{cor}[thm]{Corollary}
\newtheorem{lem}[thm]{Lemma}
\newtheorem{prop}[thm]{Proposition}
\theoremstyle{definition}

\newtheorem{rem}[thm]{Remark}

\numberwithin{equation}{section}

\newcommand{\ZZ}{\mathbb Z}

\newcommand{\PP}{\mathbb P}

\newcommand{\ra}{\rightarrow}

\newcommand{\cA}{\mathcal{A}}

\newcommand{\cH}{\mathcal{H}}
\newcommand{\cU}{\mathcal{U}_X(3, \xi)}

\newcommand{\cG}{\mathcal{G}}
\newcommand{\cM}{\mathcal{M}}
\newcommand{\cN}{\mathcal{N}}
\newcommand{\cV}{\mathcal{V}}
\newcommand{\cO}{\mathcal{O}}
\newcommand{\cR}{\mathcal{R}}
\newcommand{\cS}{\mathcal{S}}

 \DeclareMathOperator{\Gal}{Gal}
 \DeclareMathOperator{\Ker}{Ker}
  \DeclareMathOperator{\Fix}{Fix}
\DeclareMathOperator{\Pic}{Pic}
 \DeclareMathOperator{\Nm}{{Nm}}
  
 \DeclareMathOperator{\Ima}{{Im}}

 \DeclareMathOperator{\Div}{{Div}}

\begin{document}

\title[ ]{Prym varieties of triple coverings}%
\author{Herbert Lange and Angela Ortega}

\address{H. Lange\\Mathematisches Institut, Universit\"at Erlangen-N\"urnberg\\ Germany}
\email{lange@mi.uni-erlangen.de}
              
\address{A. Ortega \\ Institut f\" ur Mathematik, Humboldt Universit\"at zu Berlin \\ Germany}
\email{ortega@math.hu-berlin.de}

\thanks{}
\subjclass{14H40, 14H30}
\keywords{Prym variety, Prym map}%

\begin{abstract} 
We show that the Prym variety associated to a triple covering $f: Y \ra X $ of curves is principally polarized of dimension
$\geq 2$ if and only if $f$ is non-cyclic, \'etale and $X$ is of genus 2. We investigate some properties of these Prym varieties and their moduli.  
\end{abstract}

\maketitle

\section{Introduction}

Let $f: Y \ra X$ be a covering of smooth projective curves. The {\it Prym variety $P(f)$} associated to this covering is, by definition, the connected component containing zero of the kernel of the norm map $\Nm_f$ of the Jacobian $JY$ onto the Jacobian  $JX$. We say that $P(f)$ is a {\it principally polarized Prym variety} if the canonical principal polarization of $JY$ restricts to a multiple of a principal polarization on $P(f)$. In \cite[Proposition 12.3.3]{bl} it is claimed that  $P(f)$ is a principally polarized Prym variety of dimension at least 2 if and only if $f$ is a double covering ramified at most at 2 points and $X$ is of genus $\geq 3$. In this classification one case is missing, namely the Prym variety $P(f)$ associated to a non-cyclic \'etale triple covering $f$ of a curve $X$ of genus 2, is principally polarized of dimension 2.
The aim of this article is to fill this gap and carry out a study of these Prym varieties and their associated moduli spaces. 
\\

In the second section we recall some well known results about coverings with dihedral monodromy group. In Section 3 we prove
that for a non-cyclic triple covering, $P(f)$ is a principally polarized Prym variety of dimension $\geq 2$ if and only if $f$ is \'etale and $X$ is of genus 2. Let $\cM_2$ be the moduli space of smooth curves of genus 2 and $\cR^{nc}_{2,3}$ the moduli space of non-cyclic \'etale triple coverings of curves of genus 2. In the fourth section we show that the forgetful map  $\cR^{nc}_{2,3} \ra \cM_2$ is \'etale (at the level of the corresponding Deligne-Mumford stacks), and of degree 60. 
If $p: Z \ra Y$ is the Galois closure of $f: Y \ra X$, then the composed map $\delta \circ
f\circ p : Z \ra \PP^1 $ is Galois with Galois group the dihedral group $D_6$ of order 12 (Corollary \ref{galois}), where $\delta: X \ra \PP^1$ is the hyperelliptic covering. We use this to show that, surprinsingly, the curve $Y$ is hyperelliptic (Theorem \ref{thm4.12}). This fact allows us to describe the theta divisor $\Xi$ of $P(f)$. It turns out that $P(f)$ is a Jacobian of a smooth curve $\Xi$ (Theorem \ref{thmjac}). The curve $\Xi$ can be described explicitely in terms of the Weierstrass points of $Y$ (see Proposition \ref{eqweier}).\\

 The map $Pr: \cR_{2,3}^{nc} \ra \cA_2$ associating to any covering $f$ the principally polarized abelian surface $(P(f), \Xi)$ is called the {\it Prym map}. In Section 5 we show that $Pr$ is of degree 10 onto its image and that $\cR_{2,3}^{nc}$ is rational. Let $\alpha: Y \ra JY $ denote the Abel map  (depending on the choice of a point $y_0 \in Y$)  and $\pi: JY \ra P(f)$ the canonical projection. The composition $\alpha_f := \pi \circ \alpha: Y \ra P(f)$ is called the {\it Abel-Prym map} of $P(f)$ (also dependent on $y_0$). In Theorem \ref{thm6.3}, we show that $\alpha_f$ is injective away from four of the Weierstrass points of Y, which have the same image. \\

Given a covering $f: Y \ra X $ as above, the line bundle $\xi:= \det f_*\cO_Y$ is a 2-torsion point of $JY$. Let  $\cU$ denote the moduli space of $S$-equivalence classes of semistable vector bundles of rank 3 and determinant $\xi$. 
Considering the elements of $P(f)$ as line bundles of degree zero, the direct image defines a morphism  $f_*: P(f) \ra \cU$. In Proposition \ref{prop4.5} we show that $f_*$ is injective.\\

To any non-cyclic \'etale triple covering $f$ as above, one can associate another abelian surface in a canonical way, namely
the Prym variety $P(p,q)$ of the pair of maps $p:Z \ra Y$ and $q:Z \ra D$,  where $p$ is as above and $q$ is the natural map onto the discriminant curve $D$ of $f$. By definition $P(p,q)$ is an abelian subvariety of the Prym variety $P(p)$ (see Section 8 for the definition). In the last section, we show that the restriction of the principal polarization of $P(p)$ to $P(p,q)$ is of type (1,1).

\section{Coverings with monodromy the dihedral group}

Let $f: Y \ra X$ be a covering of degree $d$ of an irreducible smooth projective curve of genus $g_X$. Let $B$ be the finite subset of $X$ 
consisting of the branch locus of $f$ and fix a point $x_0 \in X \setminus B$. The covering induces a representation $\rho_f: \pi_1(X \setminus B,x_0) \ra \cS_{f^{-1}(x_0)}$ in the usual way, where $\cS_{f^{-1}(x_0)}$ denotes the group of permutations of the fibre
${f^{-1}(x_0)}$. Choosing an identification ${f^{-1}(x_0)} = \{1, \ldots, d\}$, we get a representation of $\pi_1(X \setminus B,x_0)$ in the symmetric group $\cS_d$ of degree $d$ which we also denote by $\rho_f$. Its image 
$$
\cM(f) : = \Ima(\rho_f)
$$
does depend a priori on the base point $x_0$ and the identification ${f^{-1}(x_0)} = \{1, \ldots, d\}$. For a different 
choice of these the corresponding image is a conjugate subgroup.  It is a transitive subgroup if and only if $Y$ is 
irreducible and called the {\it monodromy group} of the covering $f$. 

Given $X$ and 
a finite subset $B \subset X$,  it is well known that the map 
$f \mapsto \rho_f$ induces a bijection between the following sets:

(a) the set of isomorphism classes of irreducible coverings $f: Y \ra X$ of degree $d$ whose branch points lie in $B$ 
and

(b) the set of representations $\rho: \pi_1(X \setminus B,x_0) \ra \cS_d$ with transitive image up to conjugacy in $\cS_d$.

On the other hand, given an irreducible covering $f: Y \ra X$, let $Z \ra X$ denote the Galois closure of $f$; its Galois group
$$
\cG(f) := \Gal(Z/X)
$$
is called the {\it Galois group} of the covering $f$.

\begin{prop} \label{prop2.1}
An irreducible covering $f: Y \ra X$ is  Galois if and only if
$$
\deg(f) = |\cM(f)|.
$$
\end{prop}

\begin{proof}
According to \cite[Proposition p.689]{h}, for any irreducible covering $f$ the monodromy group $\cM(f)$ coincides with the
Galois group $\cG(f)$.  
Then Galois theory implies $|\cM(f)|  = |\cG(f)| = \deg(f)$ if and only if $f$ is  a Galois covering.
\end{proof} 

Consider a commutative diagram of smooth projective curves
\begin{equation} \label{diag2.1}
\xymatrix{
      & Z   \ar[dl]_{}_{p}  \ar[dr]^q \ar[dd]^h & \\
       Y \ar[dr]_f &  & D \ar[dl]^{g}\\
       & X &  \\
    }
\end{equation}
where $Z$ is the normalization of the fibre product $Y \times_X D$ of the maps $f$ and $g$. Later we assume $f$ is an \'etale
map. In that case the diagram is cartesian. The following lemma is an immediate consequence of the definitions.

\begin{lem} \label{lem2.2}
Suppose $f$ and $g$ in diagram \eqref{diag2.1} are given by representations $\rho_f : \pi_1(X \setminus B,x_0) \ra S_{f^{-1}(x_0)}$ and  $\rho_g : \pi_1(X \setminus B,x_0) \ra \cS_{g^{-1}(x_0)}$ respectively. Then the composition $h:= p \circ f: Z \ra X$
is given by 
$$
\rho_h = \iota \circ (\rho_f,\rho_g): \pi_1(X \setminus B, x_0) \ra \cS_{f^{-1}(x_0) \times g^{-1}(x_0)},
$$
where $\iota: \cS_{f^{-1}(x_0)} \times \cS_{g^{-1}(x_0)} \ra \cS_{f^{-1}(x_0) \times g^{-1}(x_0)}$ is the obvious map. \hspace{3.6cm} $\square$ 
 \end{lem}

\begin{lem} \label{lem2.3}
Let the notation be as in the previous lemma. Then $\Ima \rho_{h} \simeq \Ima \rho_f$ if and only if $\rho_g$ factors via $\rho_f$.
 \end{lem}
\begin{proof}
Suppose $\rho_g = \alpha \circ \rho_f$ with $\alpha: \cS_{f^{-1}(x_0)} \ra \cS_{g^{-1}(x_0)}$. Then 
$\rho_{h} = \iota \circ (id,\alpha) \circ  \rho_f$. Since $\iota \circ (id,\alpha)$ is injective, it gives an isomorphism 
$\Ima \rho_f \simeq \Ima \rho_{h}$.

Conversely, suppose that $\Ima \rho_{h} \simeq \Ima \rho_f$. Since $\iota$ is injective, this gives an isomorphism $\Ima (\rho_f,\rho_g) \simeq \Ima \rho_f$. If $p$ and $q$ denote the first and second projection of $\cS_{f^{-1}(x_0)} \times \cS_{g^{-1}(x_0)}$, the equation
$p \circ (\rho_f,\rho_g) = \rho_f$ implies that $p$ induces an isomorphism $\Ima (\rho_f,\rho_g) \simeq \Ima \rho_f$. Let $\beta$ denote the inverse of this isomorphism. Then $\rho_g = q \circ (\rho_f,\rho_g) = q \circ \beta \circ \rho_f$ is the asserted factorization.
\end{proof}

For a positive integer $d$, consider the dihedral group  
$$
D_d := < \sigma, \tau\;|\; \sigma^d = \tau^2 = 1, \tau \sigma \tau = \sigma^{-1}>
$$
of order $2d$, and a covering 
$$
f:Y \ra X
$$
of degree $d$ of irreducible curves with monodromy group $\cM(f) \simeq D_d$. 
Let $\rho_f: \pi_1(X \setminus B, x_0) \ra \cS_d$ denote the monodromy representation. Consider the signature map 
$sign: \cS_d \ra \cS_2$ and let 
$$
g: D \ra X
$$ 
be the covering corresponding to the composition $sign \circ \rho_f: \pi_1(X \setminus B, x_0) \ra \cS_2$. The curve
$D$ is sometimes called the {\it discriminant curve} of the covering $f$, which explains the notation $D$. With these maps we consider the commutative diagram \eqref{diag2.1}
with $\deg(f) = \deg(q) = d$ and $\deg(g) = \deg(p) = 2$.

\begin{lem} \label{lem2.4}
Under these assumptions the following conditions are equivalent:

\emph{(a)} the curve $Z$ is irreducible;

\emph{(b)} the curve $D$ is irreducible;

\emph{(c)} the image of $\rho_f$ is not contained in the alternating group $A_d$. 
 \end{lem}

\begin{proof}
By assumption $X$ and $Y$ are irreducible. This implies the equivalence of (a) and (b). The covering $g$ is of degree 2, hence $D$ is reducible if and only if $\rho_g$ is trivial. This is the case if and only if $\Ima \rho_f \subset A_d$. 
\end{proof}

\begin{prop} \label{prop2.5}
Suppose $f: Y \ra X$ is a covering of degree $d$ with Galois group $D_d$ such that $\Ima_{\rho_f} \not \subset A_d$. 
Then the covering $h = f\circ p: Z \ra X$ is the Galois closure of $f$.
\end{prop}

\begin{proof} 
By Lemma \ref{lem2.4} the curve $Z$ is irreducible.
According to Proposition \ref{prop2.1} it suffices to show that $\Ima \rho_h  \simeq D_d$.
This is clear by Lemma \ref{lem2.3}, since $\rho_g$ factors via $\rho_f$ by definition.
\end{proof}

\section{Non-cyclic coverings of degree 3}

Let $f:Y \ra X$ denote a covering of degree 3 of smooth irreducible curves. 
Since $\cS_3 = D_3$ and $A_3 = \ZZ_3$, we get as an immediate consequence of Proposition \ref{prop2.5},

\begin{prop} \label{prop3.1}
The covering $f:Y \ra X$ of degree 3 is non-cyclic if and only if the curve $Z$ of diagram \eqref{diag2.1} 
is Galois over $X$ with Galois group $\cS_3$. 
\end{prop}
As $f$ has degree 3, the ramification points of $f$ are either simple, i.e.
of order 2, or total, i.e. of order 3. Let $s$ and $t$ denote the number of simple and total ramification points of $f$ respectively.
Then the Hurwitz formula gives
$$
g_Y = 3g_X -2 + \frac{s}{2} + t.
$$ 
\begin{prop} \label{prop3.2}
Let $f:Y \ra X$ be a non-cyclic covering of degree 3 with $s$ simple and $t$ total ramification points. 

\emph{(a)} The covering $D \ra X$ is ramified exactly over the $s$ simple branch points of $f$ and 
$$
g_D = 2g_X -1 + \frac{s}{2};
$$

\emph{(b)} The covering $Z \ra D$ is cyclic and ramified exactly over the $2t$ preimages in $D$ of the $t$ total branch points of $f$ and
$$
g_Z = 6g_X -5 + \frac{3}{2}s +2t;
$$

\emph{(c)} The covering $Z \ra Y$ is ramified exactly over the s non-ramified points in the fibres over the simple branch points of $f$.
\end{prop}

\begin{proof}
A simple ramification point corresponds to a transposition in the monodromy group of $f$. This and the Hurwitz formula imply (a). 
The statement (b) follows similarly, since a total ramification point corresponds to a cycle of length 3. Part (c) is a consequence of (a) and (b)
using the commutativity of diagram \eqref{diag2.1}.
\end{proof}

The {\it Prym variety} $P(f)$ of the  covering $f: Y \ra X$ is by definition the complement of the abelian subvariety $f^*JX$ in the canonically polarized Jacobian $JY$. 

\begin{prop} \label{prop3.3}
The canonical polarization of $JY$ induces on $P(f)$ a polarization of type $(1,\cdots,1,3,\cdots,3)$, where
$1$ occurs $g_X -2 + \frac{s}{2} + t$ times and $3$ occurs $g_X$ times.
\end{prop}

\begin{proof}
The homomorphism $f^*: JX \ra JY$ is injective, since $f$ is non-cyclic. Hence the canonical polarization of $JY$ induces a polarization of type 
$(3,\cdots, 3)$ on $f^*JX$. Since $\dim P(f) = 2g_X -2 + \frac{s}{2} + t$, the assertion follows from \cite[Corollary 12.1.5]{bl}.
 \end{proof}

We call $P(f)$ a {\it principally polarized Prym variety} if the induced polarization on $P(f)$ is a multiple 
of a principal polarization.
In order to determine these Prym varities, we may assume that $X \neq \PP^1$, since in this case $P(f) = JY$.

 \begin{cor} \label{cor3.4}
Let $f: Y \ra X$ be a non-cyclic covering of degree 3 of smooth projective curves. Then $P(f)$ is a principally polarized Prym variety if and only if either\\
\emph{(i)} $X$ has genus 2 and $f$ is \'etale, or\\ 
\emph{(ii)} $X$ has genus 1 and $(s,t) = (2,0)$ or $(0,1)$.
 \end{cor}

\begin{proof}
According to Proposition \ref{prop3.3} the abelian variety $P(f)$ is a principally polarized Prym variety if and only if $g_X-2 + \frac{s}{2} +t = 0$.
This shows the assertion.
\end{proof}

Case (ii) is uninteresting, since here $P(f)$ is of dimension 1. In the next section we will study case (i) in more detail.

\section{The Prym variety of a non-cyclic 3-fold covering of a genus-2 curve}

\subsection{The set up}
Let $X$ be a curve of genus 2 and $f: Y \ra X$ a connected non-cyclic \'etale covering of degree 3. Then we have diagram \eqref{diag2.1} where  $D$ is the discriminant curve of $f$ and, according to Proposition \ref{prop3.2}, all maps in the diagram are \'etale. Hence \eqref{diag2.1} is a Cartesian diagram. According to Proposition \ref{prop3.1} the curve $Z$ is the Galois closure of $f$ with Galois group $\cS_3$. For the genera of the curves we have that
$$
g_Y = 4, \qquad g_D = 3, \qquad g_Z = 7.
$$

It is easy to give an example of such a covering $f$. For this we exhibit a surjective homomorphism $\psi: \pi_1(X,x_0) \ra \cS_3$.
Since 
\begin{equation} \label{eq4.1}
\pi_1(X,x_0) = \;\left\langle a_1,a_2,b_1,b_2 \;|\; \prod_{i=1}^2 a_ib_ia_i^{-1}b_i^{-1} =1 \right\rangle ,
\end{equation} 
it suffices to give cycles $\sigma_i, \tau_i,\;
i = 1,2$ generating $\cS_3$ and satisfying $\prod_{i=1}^2 \sigma_i \tau_i \sigma_i^{-1} \tau_i^{-1} = 1$. One can  choose 
for example $\tau_1 = \tau_2 = (1,2),\; \sigma_1 = (1,2,3)$ and $\sigma_2 = (1,3,2)$. \\  

In order to classify non-cyclic \'etale degree-3 coverings, we use the following proposition.

\begin{prop} \label{prop4.1}
There is a bijection between the following sets: 

\emph{(1)} connected non-cyclic \'etale coverings $f: Y \ra X$ of degree 3 and

\emph{(2)} \'etale Galois coverings $h: Z \ra X$ with Galois group $\cS_3$,\\
up to isomorphisms.\\
In particular, there are only finitely many non-cyclic \'etale degree-3 coverings of $X$.
\end{prop}

Here we consider only connected Galois coverings.
Note that two Galois coverings of the set (2) are isomorphic if they differ by an (inner) automorphism of $\cS_3$. Note also that any isomorphism of two coverings in (1) can be extended to an isomorphism of their Galois closures. 
 
\begin{proof}
We already associated a Galois covering $h$ to every $f$. Conversely, let $h: Z \ra X$ be a Galois covering as in (2). The group $\cS_3$ 
admits exactly 3 subgroups of index 3. Let $Y_i \ra X, \, i=1,2,3$ denote the sub-coverings corresponding to them. They 
are \'etale and non-cyclic, since the subgroups are not normal. The coverings $Y_i \ra X$ are isomorphic to each other, since the subgroups
are conjugate. The last assertion follows from the fact that there are only finitely many  \'etale Galois 
coverings of $X$ with Galois group $\cS_3$ (in fact less than $2^4 \cdot 3^6$, which is the number
of all possible double coverings of $X$ times the number of all possibles triple cyclic coverings of $D$; see Remark \ref{lastrmk}).    
\end{proof}

Let $\cM_2$ denote the moduli space of curves of genus 2 and $\cR_{2,3}^{nc}$ denote the moduli 
space of non-cyclic \'etale degree-3 coverings of curves of genus 2. It exists as a coarse moduli space, since the moduli space of 
\'etale Galois coverings of curves of genus 2 with group $\cS_3$ does. 

\begin{cor}  \label{cor4.2}
The forgetful map $\phi: \cR_{2,3}^{nc} \ra \cM_2$, $[f:Y \ra X] \mapsto [X]$ is a finite, \'etale covering of Deligne-Mumford stacks of degree 60.
\end{cor}

\begin{proof}
The finiteness of the map $\phi$ follows from Proposition \ref{prop4.1}.
We have to show that for every curve $X$ of genus 2 there are exactly 60 classes of \'etale Galois coverings over 
$X$ with Galois group $\cS_3 $.
Using Riemann's existence theorem and the fact that the group of (inner) automorphisms of $\cS_3$ is isomorphic to
$\cS_3$ itself, it suffices to show that there are exactly 360 surjective homomorphisms  
$$
\psi: \pi_1(X, x_0) \ra \cS_3.
$$  
Given a set of generators of $\pi_1(X,x_0)$ as in \eqref{eq4.1}, we denote by $A_1, A_2, B_1, B_2 $ their 
images in $\mathcal{S}_3$ under a map $\psi: \pi_1(X, x_0) \ra \cS_3$. Then it suffices to count the number 
of quadruples $(A_1,A_2,B_1,B_2) \in \cS_3^4$, whose entries generate $\cS_3$ and which verify 
\begin{equation} \label{eq4.2}
 [A_1, B_1] = [B_2, A_2],
\end{equation} 
where  $[A_i, B_i]$ is the commutator of $A_i, B_i$. Observe first that $\psi$ is surjective if the quadruple contains 
at least one element of order 3 and one transposition or 
2 different transpositions. We shall consider 4 cases depending on the number of transpositions 
among the $A_i$'s and $B_i$'s.  

Suppose first that there is only one 
transposition, say $A_1=\tau$. Then $[B_2, A_2]=1$ and relation \eqref{eq4.2}  forces $B_1$ to be the identity  and 
$A_2, B_2$ to be elements
of order 3, not both trivial.  By counting the choices for  $\tau$ and the elements of order 3 we get
$4\cdot 3 \cdot 8 =96 $ possibilities.  

Suppose now there are exactly 2  transpositions. In this case  
\eqref{eq4.2} imposes  either $(A_1, B_1) = (\tau , \tau ),  \  (A_2, B_2) = (\sigma^i , \sigma^j) $ with $\sigma^3=1$ 
and not both $A_2, B_2$ trivial, which gives $2\cdot 3 \cdot 8 =48$ possibilities,  or   
$(A_1, B_1) = (\tau , \sigma^i ),  \  (A_2, B_2) = (\tau' , \sigma^i) $  with $4\cdot 6$ choices if $\sigma^i=1$ (and then 
$\tau \neq \tau'$) and $4\cdot 9$ choices if $\sigma^i \neq 1$.  So in this case we obtain $48 + 24 + 36 = 108$ possibilities. 

Similarly, when we suppose exactly 3 transpositions in the quadruple, we have for example $(A_1,B_1)= (\tau, \tau),
  \  (A_2, B_2) = (1 , \tau') $ with $\tau \neq \tau'$, giving $3 \cdot 2^3 $ choices, or $(A_1,B_1)= (\tau, \tau'),
  \  (A_2, B_2) = (\tau'' , \sigma^i) $ with $\tau \neq \tau'$ and $\sigma^i$ determined by the other elements. 
 In this way, we get  $3^2 \cdot 2^3 $ choices. Hence we obtain $24 + 72 = 96$ different possibilities in this case. 
 
 Finally, by assuming all non-trivial elements are transpositions and there are at least 2 of them, we add 60 more 
 possibilities. Summing up the 4 cases  we get the result.
\end{proof}

\subsection{The extended diagram}

Consider again diagram \eqref{diag2.1}. We denote by $\tau_Z$ and $\tau_D$ the involutions corresponding to the double coverings
$p$ and $g$ and by $\sigma_Z $ the automorphism defining the cyclic covering $q:Z \ra D$.

As a curve of genus 2, $X$ is hyperelliptic. According to \cite[Example 1 p.64]{r} the hyperelliptic involution 
$\iota_X$ of $X$ lifts to an involution on any cyclic covering. 
So $\iota_X$ lifts to an involution $\iota_D$ on $D$ which moreover commutes with $\tau_D$ (again by \cite[Example 1 p.64]{r}). 
Hence $\iota_D$ and $\tau_D$ 
generate the Klein group $\ZZ_2 \oplus \ZZ_2$ acting on $D$ and $\kappa_D := \iota_D  \tau_D = \tau_D  \iota_D $  
is also a lift of $\iota_X$. According to \cite[Corollary of Proposition V.1.10, p.267]{fk}, one of these 
liftings, say $\iota_D$, is a hyperelliptic involution on $D$. 

In particular, $D$ is hyperelliptic and $\iota_D$ lifts to an involution $\iota_Z$ on $Z$. If $r: Z \ra B$ denotes
the corresponding double covering, we have the following commutative diagram

\begin{equation} \label{diag4.1}
\xymatrix{
        & Z  \ar[dl]_{p} \ar[dd]^{q} \ar[dr]^r & \\
       Y=Z/\tau_Z \ar[dd]_{f} & & B=Z/\iota_Z \ar[dd]  \\
        & D \ar[dl]_{g} \ar[d]  \ar[dr] &   \\
        X=D/\tau_D \ar[dr]_{\delta} & E=D/\kappa_D  \ar[l] \ar[d] \ar[r] &\PP^1=D/\iota_D\ar[dl] \\ 
       & \PP^1  & 
    }
\end{equation} 

Observe that  $\iota_Z\tau_Z\iota_Z^{-1}\tau_Z^{-1}$ acts on the fibers of the map $q: Z\rightarrow D$. So, 
$\tau_Z  \iota_Z $ and $\iota_Z \tau_Z $ differ by a power of the automorphism $\sigma_Z$, i.e. $\iota_Z \tau_Z = 
\sigma^i_Z \tau_Z \iota_Z$ for some $i$,  $0 \leq i \leq 2$. If $i \geq 1$, multiplying this equation by $\sigma^i$ we get 
$$
\sigma_Z^i \iota_Z \tau_Z = \sigma_Z^{-i} \tau_Z \iota_Z = \tau_Z \sigma_Z^i  \iota_Z .
$$
Replacing $\iota_Z$ by $\sigma_Z^i \iota_Z$  (note that with $\iota_Z $ also $ \sigma_Z^i  \iota_Z$ is an involution),
we may assume that $\iota_Z  \tau_Z  = \tau_Z  \iota_Z $.

\begin{prop} \label{prop4.6}
The group generated by the automorphisms $\iota_Z, \tau_Z$ and $\sigma_Z $ is the dihedral group $D_6$ of order 12.
\end{prop}

\begin{proof}
For simplicity, we drop the indices of the automorphisms of $Z$ in this proof. 
According to \cite[Example 1 p.64]{r} we have $\iota \sigma = \sigma^2 \iota $.
It suffices to show that $\psi  := \iota  \sigma \tau$ is of order 6. Since $\iota  \tau  = \tau  \iota $ and $\tau \sigma = 
\sigma^2 \tau$, 
$$
\psi^2 = \iota \sigma \iota \tau \sigma \tau = \sigma^2 \sigma^2 = \sigma,
$$
which implies the assertion.
\end{proof}

\begin{cor} \label{galois}
The covering $\delta \circ f \circ p: Z \ra \PP^1$ is Galois with Galois group $D_6$.
\end{cor}

\begin{proof}
For the proof just note that the covering $\delta \circ f \circ p: Z \ra \PP^1$ is Galois if and only if 
$|Aut(Z/\PP^1)| = \deg (\delta \circ f \circ p)$.
\end{proof}

\begin{cor} \label{cor4.8}
The curve $B$ is of genus 2.
\end{cor}

\begin{proof}
According to \cite[Theorem 1]{r} the Prym variety $P(q)$ of the covering $q:Z \ra D$ is isomorphic to $JB \times JB$, 
which gives the assertion, since $\dim P(q) = 4$. For another proof of this fact see \cite[Proposition 2.3]{o}.
\end{proof} 

\begin{rem}
Let $\phi_{P(q)}$ denote the polarization of $P(q)$ induced by the canonical polarization of $JZ$. If $JB$ is identified 
with its dual via the canonical polarization, then the polarization on $JB \times JB$ induced
by $\phi_{P(q)} $ via the isomorphism $P(q) \simeq JB \times JB$ of \cite[Theorem 1]{r} is given by the matrix
$
\left(
\begin{array}{cc}
2 & -1 \\
-1 & 2
\end{array} \right).
$ 
\end{rem}

\begin{prop}\label{inj}
The map $p^*: JY \ra JZ$ is injective on $P(f)$. Moreover, $p^*(P(f))$ is contained in $P(q)$. 
\end{prop}

\begin{proof}
The kernel of $p^*$ is generated by a 2-division point $\lambda \in JY$. The fact that $p$ is a lift of 
the covering $g$ means that $\lambda \in f^*(JX)$. Now $f^*(JX) \cap P(f)$ consists of 3-division points, 
so does not contain $\lambda$. This implies the first assertion.
For the last assertion,  note that $P(f)$ is the kernel of the norm map $\mbox{Nm}_f$. Hence $p^*(P(f)) \subset \Ker (\mbox{Nm}_q)$
(in fact,  $\Nm_q \circ \  p^* = g^* \circ \Nm_f$ holds).
Since $p^*(P(f))$ is connected, it is contained even in $(\Ker \mbox{Nm}_q)^0 = P(q)$; for details, see \cite[Proposition 2.2]{lr}.
\end{proof}

\subsection{The full extended diagram}

According to Corollary \ref{galois} the covering $Z \ra \PP^1$ is Galois with group $D_6$, 
for which we will use for simplicity the following notation
$$
D_6 = \langle \psi, \tau \mid \psi^6=\tau^2 = 1, \  \tau \psi \tau = \psi^{-1} \rangle,
$$
with $\tau = \tau_Z$ and $\psi =\iota_Z \sigma_Z \tau_Z $ as in the previous section (and thus $\iota_Z = \psi^3 \tau$). 
The subgroup diagram of $D_6$, which contains one representative of every conjugacy class of subgroups together with their
natural inclusions, gives us the following commutative diagram of sub-coverings of the curve $Z$. We keep 
the notation of the previous diagram. 
\begin{equation} \label{diag4.4}
\xymatrix@C=13pt@R=22pt{
       & &Z \ar[dll]_{p} \ar[dr]^{r} \ar[dd]^{q} \ar[drr]^{\alpha}  & & \\
       Y= Z/ \langle \tau \rangle \ar[ddd]_{f} \ar[ddrr]^{\gamma} & &&B= Z / \langle \psi^3\tau \rangle \ar[ddl] \ar[ddd]^\nu& 
       A = Z / \langle \psi^3 \rangle  \ar[ddd]^{\mu} \ar[ddll]\\
       & &D= Z / \langle \psi^2 \rangle \ar[ddll]_{g} \ar[ddr] \ar[ddrr]_{\beta} &&\\
       && C = Z / \langle \psi^3, \tau \rangle \ar[dd]^{\overline{f}} &&\\
       X =Z / \langle \psi^2, \tau \rangle \ar[drr]^{\delta}  & && \PP^1 = Z/ \langle \psi^3\tau,\psi^2 \rangle \ar[dl] & E=Z / \langle \psi \rangle \ar[dll] \\
      & & \PP^1 =Z / D_6 & &
    }
\end{equation} 

The maps of $Z$ to the curves in the second (respectively third, fourth and fifth) row of the diagram 
are of degree 2 (respectively 3, 4 and 6). 

\begin{prop}
{\em (a)} The curves $A$ and $E$ are of genus 1.\\
{\em (b)} The curve $C$ is of genus 0.
\end{prop}
\begin{proof}
For both parts of the proposition we use the following
Theorem of Accola \cite[Theorem 5.9]{a}:

{\it
Let $Z$ be a compact Riemann surface of genus $p$. Suppose that $Z$ admits a finite group of automorphisms $G_0$ 
and  that this 
group admits a partition (i.e. $G_0= \bigcup_{i=1}^s G_i$, where  $\{ G_i\}_{i=1}^{s}$ 
is a collection of subgroups of $G_0$ with $G_i \cap G_j = \{ e \}, \; 1 \leq i< j \leq s $ ). 
Let $n_i= |G_i |$, $X_i = X/ G_i$ and $p_i$ be the genus of $X_i$ for $i=0, \ldots, s$. Then}
$$
  (s-1)p + n_0 p_0 = \sum_{i=1}^s n_ip_i.
$$

To show (a) we apply the Theorem of Accola to the  partition $G_0 = D_6 = \cup_{i=1}^7 G_i$ with 
$$ G_1= \langle \psi \rangle,\; G_2 =  \langle \tau \rangle,\; G_3= \psi G_2 \psi^{-1},\; G_4= \psi^2 G_2 \psi^{-2},$$
$$ 
G_5 =  \langle \psi^3\tau \rangle,\; G_6= \psi G_5 \psi^{-1},\; G_7= \psi^2 G_5 \psi^{-2},
$$ to give
 \begin{eqnarray*}
42 =  6\cdot 7 + 0  & = &  6 g(Z/ \langle \psi \rangle) + 3 (2 g(Z/ \langle \tau \rangle)) + 3 (2 g(Z/ \langle \psi^3\tau \rangle)) \\
 & =  & 6g(E) +  6g(Y) + 6g(B). 
 \end{eqnarray*}
Since $g(Y) = 4$ and $g(B) = 2$ according to Corollary \ref{cor4.8}, this implies $g(E)=1$.
Consider the upper right commutative square, i.e. $\mu \circ \alpha = \beta \circ q$. Since $f$ is \'etale,
so ist $q$. From the facts that  $\deg \mu$ and $\deg \beta$  are coprime and $q$ is cyclic \'etale, we conclude that $\mu$ is \'etale. This implies $g(A) = 1$.

For the assertion (b), we apply the Theorem of Accola to the group $G_0 = \langle \psi^3,\tau \rangle$ of automorphisms of the curve $Z$ 
with partition $G_0 = \cup_{i=1}^3 G_i$ with 
$$
G_1=\langle  \psi^3 \rangle,\; G_2=\langle  \psi^3\tau  \rangle, \; G_3=\langle  \tau \rangle,
$$
to obtain the equality
$$
14 + 4g(C) = 2g(A) + 2 g(B) + 2g(Y) = 2 + 4 + 8,
$$ 
which implies the assertion.
\end{proof}

Since $\deg \gamma = \deg r = 2$, we get as an immediate consequence,

\begin{thm} \label{thm4.12}
Any non-cyclic \'etale 3-fold covering $Y$ of a curve of genus 2 is hyperelliptic.
\end{thm}

\subsection{The theta divisor of the Prym variety $P$} \label{ss4.5}
According to Theorem  \ref{thm4.12}, the curve $Y$ is hyperelliptic. Let $q_1, \dots, q_{10}$ 
denote its Weierstrass points and let $p_1, \dots, p_6$ be the Weierstrass points of $X$. 
From the commutativity of the diagram 
\begin{equation} \label{diag3}
\xymatrix@R=25pt@C=40pt {
        Y \ar[r]^{\gamma \quad}_{2:1} \ar[d]^{3:1}_f & \PP^1= C  \ar[d]^{3:1}_{\bar{f}}  \\
       X \ar[r]^{\delta \qquad}  &  \PP^1= Z / D_6 
    }
    \end{equation}
we conclude that the image of any Weierstrass point under the map $f$ is a Weierstrass 
point of $X$ and that there are 
2 points $p_i$ such that $f^{-1}(p_i)$ consists of 3 Weierstrass points and 4 points $p_i$ such that  $f^{-1}(p_i)$
contains only one Weierstrass point. Without loss of generality we may assume that 
$$
f(q_1)=f(q_2)=f(q_3)= p_1, \quad  f(q_4)=f(q_5)=f(q_6)= p_2
$$ 
and 
$$
f(q_{6+i})=p_{2+i} \quad \mbox{for} \quad  i=1,\dots, 4.
$$ 
Observe that, if $h_X$ and $h_Y$ denote hyperelliptic divisors of $X$ and $Y$, we have 
$$
f^{-1}(h_X) \sim 3h_Y.
$$
Finally, denote by $\Theta$ the canonical theta divisor in $\Pic^3(Y)$, given by the image of the map
$$
Y^{(3)} \ra Pic^3(Y), \quad y_1 + y_2 + y_3 \mapsto \cO_Y(y_1 + y_2 + y_3).
$$

In order to describe the restriction of the theta divisor of $Y$ to the Prym variety $P$, we will work in $\Pic^2(Y)$.
Let $q$ be a Weierstrass point of $Y$ and $p = f(q)$. We consider the following translate of $\Theta$:
$$
\Theta_q := \Theta - q \subset \Pic^2(Y).
$$
Let $\Nm$ denote the norm map $\Pic^2(Y) \ra \Pic^2(X)$ (and also the norm map from $\Div^2(Y)$ onto $\Div^2(X)$). In the 
following we will consider $h_X$ (respectively $h_Y$) also as an element of $\Pic^2(X)$ (respectively of $\Pic^2(Y)$).

Define $\widetilde{P}$ as the translate of the Prym variety $P$,
$$
\widetilde{P} := \Nm^{-1} (h_X)  \subset \Pic^2(Y). 
$$
We can consider $\widetilde{P}$ also as an abelian surface, since it contains a distinguished point, namely $h_Y$.
\begin{lem}
For any $y_1 + y_2 + y_3 \in Y^{(3)}$ the following conditions are equivalent\\
{\em (a)} $\cO_Y(y_1 + y_2 + y_3 - q) \in \widetilde{P} \cap \Theta_q$;\\
{\em (b)} after possibly renumbering the $y_i$ we have $f(y_3) = p$ and $f(y_2) = f(\iota_Yy_1)$. 
\end{lem}
\begin{proof}
(a) is equivalent to $\Nm(y_1 + y_2 + y_3 - q) \sim h_X$ and this occurs if and only if 
\begin{equation} \label{eq4.6}
f(y_1) + f(y_2) + f(y_3) \sim h_X + p.
\end{equation}
The linear system $|h_X + p|$ has $p$ as a base point which implies that we have $f(y_3) = p$ after possibly renumbering
the $y_i$. Then \eqref{eq4.6} reads
$$
f(y_1) + f(y_2) \sim h_X,
$$
which implies   
$$
f(y_2) = \iota_Xf(y_1) = f(\iota_Yy_1).
$$
The converse implication is obvious. 
\end{proof}
Now we choose $q = q_1$, so that $f(q_1) = p_1$ and, with the notation of the previous lemma, $y_3 \in f^{-1}(p_1) = \{q_1,q_2,q_3 \}$.
Hence, if we define for $i=1,2,3$,
$$
\widetilde{\Xi}_i := \{ \cO_Y(y_1+y_2 + q_i - q_1) \in \Pic^2(Y)\;|\; y_1,y_2 \in Y \;\mbox{with} \; f(y_2) = f(\iota_Yy_1) \}
$$
with reduced subscheme structure,
we get as an immediate consequence the following set-theoretical equality:
\begin{equation} \label{eq4.7}
\widetilde{P} \cap \Theta_{q_1} = \widetilde{\Xi}_1 \cup \widetilde{\Xi}_2 \cup \widetilde{\Xi}_3.
\end{equation} 

\begin{lem} \label{lem4.11}
For $i=1,2,3$ the scheme $\widetilde{\Xi}_i$ is the disjoint union of a complete curve $\Xi_i$ and a point. More 
precisely,
$$
\widetilde{\Xi}_i = \Xi_i \sqcup \cO_Y (h_Y + q_i - q_1).
$$ 
\end{lem}

\begin{proof}
Note first that $\widetilde{\Xi}_i = \widetilde{\Xi}_1 + q_i - q_1$ for $i = 2$ and 3.
So it suffices to prove the assertion for $i=1$. Certainly the scheme $\widetilde{\Xi}_1$ is of dimension 1, since
$\widetilde{P} \cap \Theta_{q_1}$ is a divisor. 
To see this, it suffices to show that $\widetilde{P} \not\subset \Theta_{q_1}$. For instance, 
note that $\cO_Y(q_2+q_3+q_4+q_5-h_Y) \in \widetilde{P}$,
as $\Nm(q_2+q_3+q_4+q_5-h_X) = 2p_1+2p_2-h_X \sim h_X$.
 We claim that the line bundle $\cO_Y(q_1+ q_2+q_3+q_4+q_5-h_Y)$ has no sections, hence $\cO_Y(q_2+q_3+q_4+q_5-h_Y) \not\in \Theta_{q_1}$. Since $h_Y \sim 2q_5$, $\cO_Y(q_1+ q_2+q_3+q_4+q_5-h_Y)$ has a non zero section if and only if $h^0(\cO_Y (q_1 + \cdots + q_4)) \geq 2$.
This occurs if and only if the Serre dual $h_Y^3 - q_1 - \cdots - q_4 \sim q_1 + q_2 + q_3 -q_4$ is effective,
that is there exist $p_1 , p_2 \in Y$ such that $q_1+ q_2 +q_3 \sim q_4 + p_1 + p_2$, i.e., $h^0(q_1 + q_2 +q_3)
\geq 2 $; but any linear series $g^1_3$ on an hyperelliptic curve has a fixed point (see \eqref{RR}), this leads to the contradiction 
$q_4= q_i$ for some $i \neq 4$. This proves the claim.

Clearly the point $h_Y$ is contained in $\widetilde{\Xi}_1$. Define
$$
\Xi_1 := \widetilde{\Xi}_1 \setminus h_Y.
$$
We claim that $h_Y$ is not in the closure of $\Xi_1$. In order to prove this, note that any $L \in \Xi_1$ is of the form $L = \cO_Y(y_1 + y_2)$ with
$y_1$ and $\iota_Yy_2$ in the same fibre of the map $f$. The fact that $L \neq h_Y$ means that $y_1$ and $\iota_Yy_2$ are different points
of the fibre. So, since $f$ is \'etale, $h_Y$ cannot be a limit of a sequence 
of line bundles in $\Xi_1$.

It remains to prove that $\Xi_1$ is purely one-dimensional. This will be shown in Corollary \ref{curve} below.
\end{proof}

\begin{cor} \label{corcurves}
The curves $\Xi_i$ are divisors in the Prym variety $\widetilde{P}$ and we have
$$
\widetilde{P} \cap \Theta_{q_1} = {\Xi}_1 \cup {\Xi}_2 \cup {\Xi}_3.
$$
\end{cor}

\begin{proof}
This is an immediate consequence of Lemma \ref{lem4.11} and \eqref{eq4.7} just noting that $h_Y = \cO_Y(q_1 + q_2 + q_2 - q_1) =
\cO_Y(q_1 + \iota_Yq_2 + q_2 - q_1) \in \Xi_2$ and similarly for the isolated points of $\widetilde{\Xi}_2$ and $\widetilde{\Xi}_3$.   
\end{proof}

\begin{thm}\label{genus2}
The principal polarization $\Xi$ of the Prym variety $\widetilde{P}$ is given by each of the algebraically
equivalent divisors
$$
\Xi_1 \equiv \Xi_2 \equiv \Xi_3.
$$
\end{thm}
\begin{proof}
Note first that $\Xi_i = \Xi_1 + q_i - q_1$ for $i=2$ and 3. So for $i=2,3, \; \Xi_i$ is the translate of $\Xi_1$
by the 2-division point $q_i-q_1$. It remains to show that the curve $\Xi_1$ defines the principal polarization $\Xi$.
For this note that $\widetilde{P} \cap \Theta_{q_1}$ defines the polarization $3\Xi$ and, if $n_i$ is the multilplicity 
of the curve $\Xi_i$ in $\widetilde{P} \cap \Theta_{q_1}$, the Corollary \ref{corcurves} implies
$$
3\Xi \equiv n_1\Xi_1 + n_2\Xi_2 + n_3\Xi_3 \equiv (n_1+n_2+n_3)\Xi_1.
$$
Since $\Xi$ is a principal polarization and $n_i \geq 1$ for all $i$, this is only possible if $n_1=n_2=n_3=1$ and $\Xi_1$ is also a principal polarization. This implies the assertion.
\end{proof}

\subsection{The curve $\Xi := \Xi_1$}
Recall that the scheme $\Xi := \Xi_1$ is defined by
$$
\Xi = \{ \cO_Y(y + z) \in \Pic^2(Y) \; | \; f(\iota_Yz) = f(y), \;\iota_Yz \neq y \}
$$
with reduced subscheme structure. In particular $\Xi$ is contained in the Brill-Noether locus
$$
W_2 = \{ L \in \Pic^2(Y) \;|\; h^0(L) \geq 1 \}.
$$
Recall that the 
symmetric product $Y^{(2)}$ is the blow-up of $W_2$ at the hyperelliptic line bundle $h_Y$:
$$
\rho: Y^{(2)} \ra W_2.
$$
Consider the canonical double covering $Y \times Y \ra Y^{(2)}, \; (y,z) \mapsto y+z$. 
By assumption, the map $f: Y \ra X$ is \'etale,
so for every point $y \in Y$ the fibre $f^{-1}f(y)$ consists of exactly 3 points. Let us denote these points by
$f^{-1}f(y) = \{ y, y', y''\}$;
we define
\begin{equation} \label{eq4.8}
D := \{ (y,\iota_Yy'),(y,\iota_Yy'') \in Y \times Y \;|\; y \in Y \}
\end{equation}
with reduced subscheme structure. 

\begin{lem} \label{lem4.14}
The scheme $D$ is an \'etale double covering of $Y$. In particular it is a complete curve in $Y \times Y$.
\end{lem}

\begin{proof}
Let $p_1: D \ra Y$ denote the projection onto the first factor. By definition, every point of $Y$ has exactly 
2 preimages under the map $p_1$. This implies the assertion.
\end{proof}

According to the definition, the scheme $\Xi$ does not contain the point $h_Y$. This implies that the map 
$\rho: \rho^{-1}(\Xi) \ra \Xi$ is an isomorphism. In the sequel we identifiy $\rho^{-1}(\Xi)$ with $\Xi$. 
A direct consequence of the previous lemma is 

\begin{cor} \label{curve}
The scheme $\Xi$ is a complete curve in $Y^{(2)}$ and thus also in $\widetilde{P}$.
\end{cor}

For the Weierstrass points $q_i, \; i = 7, \dots, 10$ we have
\begin{equation} \label{fibre}
 f^{-1}f(q_i) = \{ q_i, z_i, \iota_Yz_i \}
\end{equation}
with some non-Weierstrass points $z_i$. With this notation we get

\begin{lem} \label{lem4.16}
The double covering $j: D \ra {\Xi}$ induced by the double covering $Y \times Y \ra Y^{(2)}$
is ramified exactly at the 8 points $(z_i,z_i), (\iota_Yz_i,\iota_Yz_i) \in D,\; i = 7 \dots, 10$. 
\end{lem}

\begin{proof}
Note first that $(z_i,z_i) = (z_i,\iota_Y(\iota_Yz_i)) \in D$ and similarly $(\iota_Yz_i,\iota_Yz_i) \in D$.
Hence it suffices to show that $D$ intersects the diagonal of $Y \times Y$ exactly in these 8 points. For this observe
that only on the fibers over the Weierstrass points $q_i, \; i = 7, \dots, 10$  two different points of the fiber are exchanged
under the hyperelliptic involution,  namely $z_i, \iota_Y z_i $ for $i= 7, \dots, 10$. 
\end{proof}

\begin{thm} \label{thmjac}
The curve $\Xi$ is smooth and irreducible of genus 2 and the principally polarized abelian surface $(\widetilde{P},\Xi)$ 
is the Jacobian of $\Xi$.
\end{thm}

\begin{proof}
Recall that the curve $\Xi$ defines a principal polarization of the abelian surface $\widetilde{P}$. 
So either it is smooth and irreducible of genus 2 or the union of 2 elliptic curves intersecting transversally at one
point. Suppose $\Xi$ is the union of 2 elliptic curves $\Xi = \Xi^1 \cup \Xi^2$. Then, by Lemma \ref{lem4.14}, also the curve $D$ is reducible: $D = D_1 \sqcup D_2$
with $D_i \simeq Y$ and thus of genus 4 for $i = 1,2$. Then without loss of generality and according to Lemma
\ref{lem4.16}, we may assume that $D_1 \ra \widetilde{\Xi}^1$ is a double covering ramified in $\leq 4$ points of $D_1$. This contradicts  the Hurwitz formula. 
\end{proof}

Recall that the curve $\Xi$ is given by
$$
\Xi = \{ [y + \iota_Yz] \in Y^{(2)} \;|\; f(y) = f(z),\; y \neq z \}
$$
with reduced subscheme structure. The involution $\iota_Y$ induces the hyperelliptic involution on $\Xi$ by
$$
\iota_{\Xi}: \Xi \ra \Xi,\quad [y + \iota_Yz] \mapsto \iota_Y([y + \iota_Yz]) = [z + \iota_Yy].
$$ 
So the following proposition is immediate.
\begin{prop} \label{eqweier}
If $\{q_1, \dots, q_{10}\}$ denote the Weierstrass points of $Y$ as above, i.e. with $f^{-1}(p_1) = \{q_1,q_2,q_3\}$ and 
$f^{-1}(p_2) = \{q_4,q_5,q_6\}$, then 
\begin{equation*} 
\{[q_1 + q_2],[q_1+q_3],[q_2+q_3],[q_4+q_5],[q_4+q_6],[q_5+q_6]\}
\end{equation*}
are the Weierstrass points of $\Xi$.
\end{prop}

\section{The Prym map}

Let $\cA_2$ denote the moduli space of principally polarized abelian surfaces. 
The {\it Prym map} for non-cyclic triple coverings is the map
$$
Pr: \cR_{2,3}^{nc} \ra \cA_2, \qquad [f: Y \ra X] \mapsto P(f).  
$$ 
According to Theorem \ref{thmjac} the image of the Prym map is contained in the open subset of $\cA_2$ consisting of Jacobians of smooth curves of genus 2.

\begin{thm} \label{thmpr}
The Prym map $Pr: \cR_{2,3}^{nc} \ra \cA_2$ is finite of degree 10 onto its image.
\end{thm}

\begin{proof}
Let $[f:Y \ra X] \in \cR_{2,3}^{nc}$ and suppose $P(f) = J\Xi$. We have the following commutative diagram  
\begin{equation} \label{diag}
\xymatrix{
        Y \ar[d]^{3:1}_f  \ar[r]^{\gamma}  & \PP^1 \ar[d]_{\overline{f}} &  \Xi \ar[l]_{\varphi} \ar[dl]^{\psi} \\
        X \ar[r]_{\delta}  &  \PP^1 &  
    }
 \end{equation}
where the left hand square is diagram \eqref{diag3}, $\varphi$ is the hyperelliptic 
covering and $\psi$ the composed map. Let the notation of the Weierstrass points of $X$ and $Y$ be as above. Then 
the Weierstrass points of $\Xi$ are given by Proposition \ref{eqweier}. The map $\psi$ is given by a pencil $g^1_6 \subset |3K_{\Xi}|$
whose ramification divisor consists of the 6 Weierstrass points of $\Xi$ and the 8 preimages of the 4 ramification points over
$\delta(p_3), \dots, \delta(p_6)$. 

Conversely, let $\Xi$ be a smooth curve of genus 2 with Weierstrass points $\{w_1, \dots, w_6 \}$. Consider the pencil 
$g^1_6 \subset |3K_{\Xi}|$ generated by 2 divisors of the form $2w_i + 2w_j + 2w_k$ where the union of the $w_i$'s 
in both divisors equals $\{w_1, \dots, w_6\}$. Hence the corresponding 6:1 covering $\psi: \Xi \ra \PP^1$ factors via the 
hyperelliptic covering, because the generating  divisors are sums of divisors linear equivalent to $K_{\Xi}$. So we have the commutative triangle of diagram \eqref{diag}.

Let $x_1$, $x_2$ be the branch points of $\psi$ in $\PP^1$ corresponding to the two generating divisors of the $g^1_6$.  
Thus  $\overline{f}$ is \'etale over $x_1$ and $x_2$.  Assume for the moment that $\overline{f}$ is simply 
ramified, say over the points $x_3,\dots,x_6 \in \PP^1$. Let $\delta: X \ra \PP^1$ be the double covering ramified in $x_1, \dots, x_6$ and let $f:Y \ra X$ be the normalization of the pull-back of $\overline{f}$ by $\delta$. Clearly $f$ is \'etale. Moreover, we claim that $f$ is non-cyclic. If $f$ was cyclic the corresponding automorphism of order 3 would permute the Weierstrass points of $Y$; but since there are 10 of them, this automorphism would have fixed points, contradicting the fact that $f$ is \'etale. So 
we are in the situation of diagram \eqref{diag}. In particular, we have $P(f) = J\Xi$.
Clearly every covering $f$ with $Pr(f) =J\Xi$ arises in this way. Moreover, $f$ is uniquely determined by the choice of 
a partition of $\{ w_1, \dots, w_6 \}$ into 2 subsets of three elements. Hence, under the above assumption, we conclude that there are exactly $\frac{1}{2} {6 \choose 3} = 10$ elements in the fibre $Pr^{-1}(J\Xi)$. 

Observe that if $\overline{f}$ is not simply ramified, then either it has 2 triple ramification points, or it has 1 triple ramification point and 2 simple ones.  In the former case  $\psi$ is ramified over 4 branch points, in the last case $\psi$ ramifies over 5 points. We claim that a general $\Xi$ does not admit such coverings  over $\PP^1$. This will show that  for a general $\Xi$ all 10 pencils satisfy the assumption above. To prove the claim, we estimate the dimension of the following locus in $\cM_2$
$$
\cN_b := \{ [C] \in \cM_2 \mid C \mbox{ has a base point free }  g_6^1 \mbox{ branched over $b$ points }\}, 
$$ 
for $b=4,5$. Let $\cH_{6,b}$ denote the Hurwitz scheme of coverings $\psi : C \ra \PP^1 $ of degree 6 branched over $b$ points. It is known
that the Hurwitz schemes have expected dimension $b$, the number of branch points. 
Hence the dimension of $\cN_b$, for $b=4,5$, is at most $\dim \cH_{6,b} - \dim $ Aut$(\PP^1) \leq 2 $, which is less than the dimension of $\cM_2 $. This completes the proof.

\end{proof}

\begin{prop}
The moduli space $\cR^{nc}_{2,3}$ is rational. 
\end{prop}
\begin{proof}
Consider the diagram \eqref{diag}. The moduli space $\overline{\cV}$ of triple coverings $\bar{f}: \PP^1 \ra \PP^1$ 
is a rational curve according to \cite[pp. 96-98]{eehs}. Let $\cV$ denote its open subset consisting of triple coverings 
branched over 4 distinct points.
Let $B(\bar{f}) \subset \PP^1$ be the branch locus of $\bar{f}$.  A pair $(p_1 +p_2, \bar{f})$
in a Zariski open subset of $(\PP^1)^{(2)}  \times \cV $ determines uniquely a genus 2 
curve $X$ whose hyperelliptic covering
$\delta$ has branched locus $B(\bar{f}) \cup \{p_1, p_2\}$ and a triple covering 
$f: Y \ra X$ defined as the normalization of the pullback of $\bar{f}$ by $\delta$. 
Conversely, it is clear from what we have said above (see diagram \eqref{diag}), that every element $[f: Y \ra X] \in \cR^{nc}_{2,3}$ defines a point 
in $(\PP^1)^{(2)}  \times \cV$. 
Thus we conclude that $\cR^{nc}_{2,3}$ is birational 
to $\PP^3$.
\end{proof}

\section{The Abel-Prym map}
Let $\alpha_{y_0}: Y \ra JY$ denote the Abel map with respect to a point $y_0 \in Y$ and $\pi: JY \ra P$ the canonical projection.
The composition 
$$
\alpha_f: Y \stackrel{\alpha_{y_0}}{\ra } JY \stackrel{\pi}{\ra} P
$$
is called the {\it Abel-Prym map} of $P = P(f)$.

Recall that every endomorphism of the Jacobian $JY$ is given by a correspondence on $Y$. In order to study the Abel-Prym map $\alpha_f$,
we need a correspondence inducing $\pi: JY \ra P$. For this we use the following notation: if $y$ is a point of $Y$, let $y'$ and $y''$
denote the other 2 points of the fibre $f^{-1}(f(y))$.

\begin{lem} \label{lem4.3}
The projection $\pi: JY \ra P$ is induced by the following correspondence on $Y$: 
$$
y \mapsto 2y - y' - y''.$$
\end{lem}  
\begin{proof}
$P$ is the complement of the abelian subvariety $f^*JX$ in the principally polarized abelian variety J$Y$. The projection 
$JY \ra f^*JX$ is the image of the trace correspondence $T: y \mapsto y + y'+ y''$. According to \cite[Section 5.3]{bl}, if an abelian subvariety of
exponent $e$ is induced by a correspondence $D$, the complementary abelian subvariety is the image of $e\cdot id_{JY} - D$. This implies the assertion, since $f^*JX$ is of exponent 3 in $JY$.  
\end{proof}

\begin{lem} \label{lem4.4}
\emph{(1)} For any 2 points $y_1,y_2 \in Y$,
$$
\alpha_f(y_1) = \alpha_f(y_2) \quad \Leftrightarrow  \quad 2y_1 + y_2'+ y_2'' \sim 2y_2 + y_1' + y_1''.
$$ 
\emph{(2)} For 2 points $y, y'$ in a fibre of $f$,
$$
\alpha_f(y) = \alpha_f(y') \quad \Leftrightarrow  \quad 3y \sim 3y'.
$$
\end{lem}

\begin{proof}
Lemma \ref{lem4.3} implies that $\alpha_f(y_1) = \alpha_f(y_2)$ if and only if $2y_1 - y_1' - y_1'' \sim 2y_2 - y_2' - y_2'' $, which 
gives (1). Inserting $y_1 = y$ and $y_2 = y'$, (1) gives $2y + y + y'' \sim 2y' +y' + y''$ which implies (2). 
\end{proof}
Recall the notation of the Weierstrass points of $Y$ and $X$ from subsection \ref{ss4.5}. In particular, according to \eqref{fibre},
$f^{-1}f(q_i)= \{q_i, z_i , \iota_Y z_i \}$ for the Weierstrass points $q_i, \; i= 7, \ldots, 10 $ with non-Weierstrass points 
$z_i$. Then we have
\begin{thm} \label{thm6.3}
 The Abel-Prym map $\alpha_f : Y \ra P $ is injective away from the Weierstrass points $q_7, \ldots, q_{10}$, which 
have the same image.
\end{thm}
\begin{proof}
According to Theorem \ref{thm4.12} the curve $Y$ is hyperelliptic. It is a consequence of the geometric version of the Riemann-Roch theorem that any complete linear system $g_d^r$ is of the form  (see \cite{acgh}, p. 13)
\begin{equation}\label{RR}
rh_Y + a_1 + \cdots + a_{d-2r}.
\end{equation}

Suppose now $y_1, y_2 \in Y$, with $\alpha_f(y_1) = \alpha_f(y_2)$. If $f(y_1)=f(y_2)$ then, according to Lemma \ref{lem4.4},
$3y_1 \sim 3y_2$, which gives $y_1=y_2$ by equation \eqref{RR}.

If $f(y_1) \neq f(y_2)$,  Lemma \ref{lem4.4} implies $2y_1 + y_2'+ y_2'' \sim 2y_2 + y_1' + y_1''$. If this linear system 
is a complete $g^1_4$ then, by \eqref{RR} it has two fixed points, which implies $y_1=y_2$. So suppose it is a subsystem of 
a two-dimensional linear system $g^2_4$. Then it follows from \eqref{RR} that $y_1$ and $y_2$ are Weierstrass points and $y_1''= \iota_Y y_1' $ as well as $ y_2''= \iota_Y y_2' $. This implies 
$$
y_1, y_2 \in \{ q_7, \ldots , q_{10}\}.
$$ 

Conversely, if $y_1, y_2 \in \{ q_7, \ldots , q_{10}\}$ then, according to \eqref{fibre}, we have $y_i''= \iota_Y y_i'$ for $i=1,2$. Now Lemma \ref{lem4.4} implies that $\alpha_f(y_1) = \alpha_f(y_2)$.
\end{proof}

\section{Direct images of line bundles on $Y$}

According to \cite[Proposition 4.7]{m} 
$$f_*\cO_Y = \cO_X \oplus E, 
$$ 
with $E$ a vector bundle of rank 2 such that $(\det E)^2 = \cO_X$.
In the case of a cyclic triple covering $f$ given by a line bundle $\eta \in JX[3] \setminus \{0\}$
it is well known that $E \simeq \eta \oplus \eta^2$. In the non-cyclic case we have

\begin{lem}
The vector bundle $E$ is stable.
\end{lem}
\begin{proof}
Consider the commutative diagram \eqref{diag2.1} and let $h:Z \ra X$ denote the composition $h = f\circ p = g \circ q$.
Let $\eta \in JD[3] \backslash \{0 \}$ and $\lambda \in JX[2]\backslash \{0 \}$ be the line bundles associated to the 
coverings $q$ and $g$. So $\Ker q^* =  \langle \eta \rangle $, $\Ker g^* =  \langle \lambda \rangle $. Since $f^* : JX \ra JY$ is injective, $|\Ker h^*| \leq |\Ker p^*|=  2 $.  If $\eta $ was in the image of $g^*$ then the inverse images of  $\eta$ and $\eta^2$ would be in the kernel of $h^*$, but
$\Ker h^*$ contains at most one non-trivial element. So, $\eta \notin \Ima g^*$.  On the other hand, by flat base change
\begin{eqnarray*}
\cO_D \oplus \eta \oplus \eta^2 & = & q_* \cO_Z \\
& = & q_* p^* \cO_Y \\
& = & g^* f_* \cO_Y \\
& = & \cO_D  \oplus g^* E. 
\end{eqnarray*}
Thus $g^* E \simeq \eta \oplus \eta^2$. The semistability of $E$ follows from the fact that $f_*\cO_Y$ is semistable (see \cite[Proof of Proposition 4.1]{b}) and $E$ is of degree zero. Suppose that $E$ has a line subbundle of degree  0, $L_1 \hookrightarrow E$ . Without lost of generality, we might assume that the composition map
 $$
 g^*L_1  \hookrightarrow g^*E  \simeq \eta \oplus \eta^2 \ra \eta 
 $$
 is non zero, where the last arrow is the projection. Since $\eta$ is also of degree 0, this map gives an isomorphism $g^*L_1 \simeq \eta $, which contradicts the fact 
 $\eta \notin g^*(JX)$. Therefore, $E$ is stable.
\end{proof}

As a consequence of the proof, we can say more about the determinant of $E$:

\begin{cor} 
Let $\lambda$ be the 2-division point defining the covering $g:D \ra X$. Then $\det E \in \{ \cO_X, 
\lambda \}$.
\end{cor}
\begin{proof}
In the proof of the previous lemma, we saw that $g^*E = \eta \oplus \eta^2$, where $\eta$ is the 3-division point of $JD$
defining the covering $q:Z \ra D $. This implies that 
$$
g^* \det E = \det g^*E = \cO_D
$$
which implies the assertion.
\end{proof}

Let $\xi$ denote the 2-division point $\det E$ of $\Pic^0(X)$ and let $\cU$ denote the moduli space of 
S-equivalence classes of semistable 
rank 3 vector bundles over $X$ with determinant $\xi$. 
The direct image $f_*L$ of a line bundle $L$ on $Y$ is vector bundle of rank 3 on $X$ and by \cite[Proof of 
Proposition 4.1]{b} it is semi-stable. Moreover,  for any line bundle $L$ on $Y$ the determinant of $f_*L$ is given by the following formula
\begin{equation} \label{detformula}
\det f_*L = \Nm L \otimes \det f_*\cO_Y 
\end{equation}
(observe that the formula is trivial for $L= \cO_Y$, then one can apply induction on the degree of $L$ using the fact that any line bundle $L'$
fits in  an exact  sequence $0 \ra L \ra L' \ra \cO_{p} \ra 0 $, where $\deg L =\deg L' -1$ and $\cO_{p}$ is the 
skyscraper sheaf supported at
the point  $p \in Y$). Thus  $f$  induces a morphism $f_*: P(f) \ra \cU$. 


\begin{prop} \label{prop4.5}
The direct image morphism 
$$
f_*: P(f) \ra \cU 
$$ is injective. 
\end{prop}

\begin{proof}

Recall the diagram \eqref{diag2.1}. According to Proposition \ref{prop3.1}, the curve $Z$ is Galois over $X$ with Galois group 
$D_3 = \langle \sigma, \tau \;|\; \sigma^3= \tau^2 =1, \tau \sigma \tau = \sigma^2 \rangle$.
For $i = 0,1,2$ denote $Z_i = Z$ and define $p_0 = p: Z_0 \ra Y, \; p_1 = p \circ \sigma : Z_1 \ra Y$ and 
$p_2 = p \circ \sigma^2: Z_2 \ra Y$.  
Then the following diagram is cartesian
\begin{equation} \label{diag5.1}
\xymatrix{
        \bigsqcup_{i=0}^2 Z_i \ar[d]_{h'}  \ar[r]^{ \quad f'}  & Z \ar[d]^h \\
        Y \ar[r]_f  &  X  
    }
 \end{equation}  
where $f'$ is the identity on each $Z_i$ and $h'|_{Z_i} = p_i$ for $i = 0,1,2$. 

Suppose now that $L, L' \in P(f)$ with $f_*L \simeq f_*L'$. Applying
flat base change to diagram \eqref{diag5.1}, we get 
$$
\oplus_{i=0}^2 p_i^*L = {f'}_*{h'}^* L \simeq h^*f_* L = h^*f_* L' \simeq {f'}_*{h'}^* L' = \oplus_{i=0}^2 p_i^*L'. 
$$
This implies 
$$
p^* L = p_0^*L \simeq p_i^*L' =(\sigma^i)^*p^*L'
$$ 
for some $i \in \{0,1,2 \}$. Hence, since $p^*|_{P(f)}$ is injective 
according to Proposition \ref{inj}, it suffices to show that $p^*L \simeq (\sigma^i)^*p^*L'$ implies $p^*L \simeq p^* L'^*$ for 
$i = 1$ and 2. Moreover, replacing $\sigma$ by $\sigma^2$, we may assume $i=1$.

So assume that $p^* L = \sigma^*p^*L'$ for some $L,L' \in P(f)$. It suffices to show that this implies
\begin{equation}
p^*L \in \Fix(\sigma^*),
\end{equation}
since then $p^*L' \simeq (\sigma^2)^* p^*L = p^*L$ and thus $L \simeq L'$.

Now the assumption means that the pair $(L,L'^{-1})$ is in the kernel of the morphism
$$
r: P(f) \times P(f) \ra JZ, \quad (L,L') \mapsto p^*L \otimes \sigma^* p^*L'.
$$  
It is well known that $r$ is a $D_3$-equivariant isogeny onto its image (for the first proof of this see \cite[Theorem 5.4]{rr}).
Here the action on the left  hand side is the standard
action of $D_3$ given by  
$$
\tau \mapsto \tilde{\tau}=\left( 
\begin{array}{ll}
1 & -1 \\
0 & -1
\end{array}
\right) , \qquad
\sigma \mapsto \tilde{\sigma} = \left( 
\begin{array}{ll}
0 & -1 \\
1 & -1
\end{array}
\right) 
$$  
and the action on the right hand side is the action induced on $JZ$ from $Z$. In particular the kernel
of the homomorphism $r$ is a $D_3$-invariant subgroup of $P(f) \times P(f)$. 

By hypothesis $(L, L'^{-1}) \in \Ker(r)$. So also
$\tilde{\sigma}( L, L'^{-1})= (L' , L \otimes L') \in \Ker(r)$ and  $\tilde{\tau}( L, L'^{-1})= ( L \otimes L', L') \in \Ker(r)$.
Together
with the assumption $p^*L = \sigma^*p^*L'$ (or equivalently  $(\sigma^2)^*p^*L = p^*L'$) this gives the equalities
$$
(\sigma^2)^*p^*L \otimes \sigma^*p^*L \otimes p^*L =0, \qquad   (\sigma^2)^*p^*L  \otimes( p^*L)^2 = 0.
$$
These equations imply 
$$
\sigma^* p^*L \simeq (\sigma^2)^*p^*L^{-1} \otimes p^*L^{-1} \simeq p^*L 
$$
which was to be shown.


\end{proof}

\section{The Prym variety of the pair $(p,q)$}
Let $f:Y \ra X$ be a non-cyclic \'etale degree 3 covering of a curve $X$ of genus 2 and consider the associated diagram \eqref{diag2.1}. 
According to \cite{lr} the {\it Prym variety $P(p,q)$ of the pair of coverings} $(p,q)$ is defined to be the complement of the abelian
subvariety $q^*P(g)$ in the Prym variety $P(p)$ with respect to the restriction of the canonical polarization of $J(Z)$. 

\begin{prop} \label{prop6.1}
$P(p,q)$ is an abelian surface. 
\end{prop}

\begin{proof}
$\dim P(p,q) = \dim JZ - \dim JY - \dim P(g) = 7 - 4 - 1 = 2.$ 
\end{proof}

Hence we have associated to the abelian surface $P(f)$ another abelian surface $P(p,q)$. 
In this section we denote  $\Xi$ the principal polarization of the Prym variety $P(p)$. Recall the curves $A$ and $B$ of 
diagram \eqref{diag4.4} and let $\Theta_A$ and $\Theta_B$ 
denote the canonical principal polarizations of their Jacobians $JA$ and $JB$. Moreover, we identify the elliptic curve
$A$ with its Jacobian.

\begin{thm} \label{thm8.2}
{\em (a)} There is a canonical isomorphism of principally polarized abelian varieties 
$$
(P(p),\Xi) \simeq (A,\Theta_A) \times (JB, \Theta_B).
$$
{\em (b)} If we consider $JB$ as an abelian subvariety of $JZ$, we have an equality of polarized abelian varieties
$$
(P(p,q),\Xi|_{P(p,q)}) = (JB, \Theta_B).
$$

\end{thm}

\begin{proof}
(a): The Theorem of Mumford on Prym varieties of \'etale double coverings of hyperelliptic curves 
(see \cite[Page 346]{mu}) applied to the following subdiagram of the full extended diagram
\begin{equation} \label{sd1}
\xymatrix@C=13pt@R=22pt{
       &Z \ar[dl]_{p} \ar[d]^{r} \ar[dr]^{\alpha}  & \\
       Y  \ar[dr]^{\gamma} & B \ar[d] \ar[d]& A  \ar[dl]\\
       & C = \PP^1 & 
    }
\end{equation}
gives an isomorphism of principally polarized abelian varieties
\begin{equation} \label{eq2}
(P(p),\Xi) \simeq (JA \times JB, \Theta_A \times JB + JA \times \Theta_B) = (A,\Theta_A) \times (JB, \Theta_B).
\end{equation}

(b): Let $\Xi_g$ denote the principal polarization of the Prym variety $P(g)$.
The theorem of Mumford applied to the following subdiagram of the full extended diagram
\begin{equation} \label{sd2}
\xymatrix@C=13pt@R=22pt{
       &D \ar[dl]_{g} \ar[d] \ar[dr]_{\beta} &\\
       X  \ar[dr]^{\delta}  &  \PP^1 \ar[d] & E \ar[dl] \\
      &  \PP^1 =Z / D_6 & 
    }
\end{equation}
yields
\begin{equation} \label{eq4}
(P(g),\Xi_g) \simeq (JE \times J\PP^1, \Theta_E \times J\PP^1 + \Theta_{\PP^1} \times JE) = (E, \Theta_E). 
\end{equation}
where we use $J\PP^1 =0$ and identify $E = JE$.

The full extended diagram consists of diagrams \eqref{sd1} and \eqref{sd2} together with the maps $f,q,\nu,\mu$ 
and $\overline{f}$ from \eqref{sd1} to \eqref{sd2}. Since the Theorem of Mumford is certainly compatible with pull-backs
with respect to $f,q,\nu,\mu$ and $\overline{f}$, we get that $q^*$ respects the decompositions \eqref{eq2} and \eqref{eq4}.
In particular we get 
$$
q^*(P(g)) =  \mu^*(E) \times \nu^*(J_{\PP^1}) = A \times \{ 0 \} = A.
$$
Since $\alpha$ and $r$ are ramified, we can consider $A$ and $JB$ 
as abelian subvarieties in $JZ$ and thus of the Prym variety $P(p)$. 
Then \eqref{eq2} implies that the complement of the abelian subvariety $A$ in $P(p)$ 
with respect to the polarization $\Xi$ is $JB$.

Now let $\Theta$ denote the canonical principal polarization of the Jacobian $JZ$. Since 
\begin{equation} \label{eq5}
\Theta|_{P(p)} = 2 \Xi,
\end{equation}
the abelian variety $JB$ is also the 
complement of $A$ with respect to the polarization $\Theta|_{P(p)}$.

On the other hand, the Prym variety $P(p,q)$ of the pair of maps $(p,q)$ is defined to be the complement of the abelian subvariety 
$q^*P(g) = A$ in $P(p)$ with respect to the polarization $\Theta|_{P(p)}$. Since the complement of an abelian subvariety
with respect to a polarization is uniquely determined, we get $P(p,q) = JB$. The equality of the polarizations is a consequence of \eqref{eq2}. 
\end{proof}

We want  to determine the $3$-division point $\eta$ of the Jacobian $JD$ inducing the 
cyclic covering $q: Z \ra D$ of the diagram \eqref{diag2.1}. More generally, consider $f:C \ra D$ a cyclic \'etale 
covering of prime degree $p$ and denote by $\Theta_C$ and $\Theta_D$ the canonical polarizations. The associated pull-back map of line bundles $f^*:JD \ra JC$ is an isogeny
onto its image with kernel a cyclic subgroup of order $p$ generated by an element $\eta \in JD$. With this setting
we have the following  lemma.

\begin{lem} \label{lem6.2}
Let $A \subset JD$ be an abelian subvariety with $\Theta_D|_A$ of type $(d_1, \ldots , d_r)$. Then the polarization
$$
\Theta_C |_{f^*A} \; \mbox{is of type} \; \left\lbrace  \begin{array}{c}
                                                       (pd_1, \ldots , pd_r)\\
                                                       (d_1, pd_2, \ldots , pd_r)
                                                      \end{array} \right. 
\mbox{if} \quad \begin{array}{c} 
           \eta \not \in A\\
           \eta \in A.
          \end{array} 
$$
\end{lem}

\begin{proof}
According to \cite[Proposition 12.3.1]{bl}, $(f^*)^*\Theta_C \equiv p \Theta_D$. This implies that  
$(f^*)^*\Theta_C|_A $ $\equiv p \Theta_D|_A$
is of type $(pd_1, \ldots , pd_r)$. Since $f^*|_A: A \ra f^*(A)$ is an isomorphism if $\eta \not \in A$ 
(here we use that $p$ is a prime) and an isogeny of degree $p$ if $\eta \in A$, this implies the assertion.
\end{proof} 

\begin{prop} \label{prop8.4}
The $3$-division point $\eta$ of $JD$ corresponding to the \'etale covering $q:Z \ra D$ is contained in the elliptic curve
$P(g) = E \subset JD$.
\end{prop}

\begin{proof} Recall that $\Theta$ and $\Xi$ are the canonical principal polarizations of $JZ$ and $P(p)$.
Since $g: D \ra X$ is an \'etale double covering, 
$\Theta_D|_{P(g)}$ is of type $(2)$. Hence by Lemma \ref{lem6.2} the polarization 
$$
\Theta|_{q^*(P(g))} \quad \mbox{is of type} \quad \left\lbrace  \begin{array}{c}
                                                         (6)\\
                                                         (2)
                                                        \end{array} \right.\quad  \mbox{ if } \quad \begin{array}{c}
                                                                                \eta \not \in P(g)\\
                                                                                 \eta \in P(g).
                                                        \end{array}
$$
Now $q^*(P(g)) \subset P(p)$ and $\Theta|_{P(p)} = 2 \Xi$. This yields 
$$
\Xi|_{q^*(P(g))} \quad \mbox{is of type} \quad \left\lbrace  \begin{array}{c}
                                                         (3)\\
                                                         (1)
                                                        \end{array} \right.\quad  \mbox{ if } \quad \begin{array}{c}
                                                                                \eta \not \in P(g)\\
                                                                                 \eta \in P(g).
                                                        \end{array}
$$
Since $\Xi$ is a principal polarization and $P(p,q)$ is the complement of $g^*(P(g))$ with respect to $\Xi_{P(p)}$, it follows that
$$
\Xi|_{P(p,q)} \quad \mbox{is of type} \quad \left\lbrace  \begin{array}{c}
                                                         (1,3)\\
                                                         (1,1)
                                                        \end{array} \right.\quad  \mbox{ if } \quad \begin{array}{c}
                                                                                \eta \not \in P(g)\\
                                                                                 \eta \in P(g).
                                                        \end{array}
$$
But we know from Theorem \ref{thm8.2} that $\Xi_{P(p,q)}$ is of type $(1,1)$. Hence $\eta \in P(g)$. 
The equality $P(g) = E$ is \eqref{eq4}. 
\end{proof}

\begin{rem}
 According to \cite[Proposition 2.4]{lr} the Prym variety $P(p,q)$ can also be defined as the complement of $p^*P(f)$ in the Prym variety $P(q)$. So we have an isogeny 
$$
P(q) \sim P(f) \times P(p,q).
$$
Recall that $P(f)$ is a principally polarized Prym variety and according to Theorem \ref{thm8.2} (b) $P(p,q)$ is also principally polarized. On the other hand, the restriction of the canonical principal polarization of $JZ$ to $P(q)$ is  of type $(1,1,3,3)$ and thus  the above isogeny is not an isomorphism. Moreover,  since
$P(q) \simeq JB \times JB$  and $P(p,q) = JB$ by Theorem \ref{thm8.2} (b), one deduces that $P(f)$ is isogenous to  $JB$.
\end{rem}

\begin{rem} \label{lastrmk}
Observe that the number of \'etale (connected) double coverings $g:D \ra X$ is $2^{4}-1=15$. On the 
other hand,  according to Proposition \ref{prop8.4}, the \'etale triple coverings $q: Z \ra D$ corresponding to the  Galois closure of $f$ are determined by a subgroup of order 3 of $E[3]$  and there are $(3^2-1)/2=4$ 
of them. Therefore, for every genus 2 curve $X$ there are $15\cdot 4 =60 $ choices for the Galois covering $Z \ra X$, which agrees with Corollary \ref{cor4.2}.  
\end{rem}


\begin{thebibliography}{999999}

\bibitem[A]{a}
    R. Accola:
    \textit{Topics in the Theory of  Riemann Surfaces}.
    Lectures Notes in Mathematics 1595, Springer - Verlag (1994).
\bibitem[ACGH]{acgh}
    E. Arbarello, M. Cornalba, P. A. Griffiths, J. Harris :
    \textit{Geometry of algebraic curves},  Volume I.
    Grundlehren der Math. Wiss. 267, Springer - Verlag (1985).
\bibitem[B]{b}
    A. Beauville:
    \textit{On the stability of the direct image of a generic vector bundle}.
    Preprint (available in http://math.unice.fr/~beauvill/ ).
\bibitem[BL]{bl}
    Ch. Birkenhake, H. Lange:
    \textit{Complex Abelian Varieties}. Second edition,
    Grundlehren der Math. Wiss. 302, Springer - Verlag (2004).
\bibitem[EEHS]{eehs}
    D. Eisenbud, N. Elkies, J. Harris, R. Speiser.
    \textit{On the Hurwitz scheme and its monodromy}.
    Comp. Math, 77 (1991), 95-117.
\bibitem[FK]{fk}
    H. Farkas, I. Kra:
    \textit{Riemann Surfaces}. Second edition,
     Springer - Verlag (1992).
\bibitem[H]{h}
    J. Harris:
    \textit{Galois groups of enumerative problems}.
    Duke Math. Journ. 46 (1979), 685-724.
\bibitem[LR]{lr}
    H. Lange, S. Recillas:
    \textit{ Prym varieties of pairs of coverings}.
    Adv. in Geometry 4 (2004), 373-387. 
\bibitem[M]{m}
    R. Miranda:
    \textit{Triple covers in algebraic geometry}. 
    Am. Journ. Math. 107 (1985), 1123-1158.
    
\bibitem[Mu]{mu}
D.~Mumford:
\textit{Prym varieties {I}}.
\newblock In L.V. Ahlfors, I.~Kra, B.~Maskit, and L.~Niremberg, editors, {\em
  Contributions to Analysis}. Academic Press (1974), 325-350.
    
\bibitem[O]{o}
    A. Ortega:
    \textit{Vari\'et\'es de Prym associ\'ees aux rev\^etements n-cycliques d'une courbe hyperelliptique}.
    Math. Z. 245  (2003), 97--103.
\bibitem[RR]{rr}
    S. Recillas, R. Rodr\'{\i}guez:
    \textit{Jacobians and representations of $\cS_3$}. 
    Aportaciones matem\'aticas 13 (1998),   117-140.
\bibitem[R]{r}
    J. Ries: 
    \textit{The Prym variety for a cyclic unramified cover of a hyperelliptic Riemann surface}.
    J. Reine Angew. Math. 340 (1983), 59-69. 

\end{thebibliography}
\end{document}